\documentclass[manmat]{svjour}
\usepackage{latexsym}
\usepackage{graphics}
\usepackage{amsmath}
\usepackage{amscd}
\usepackage{amssymb}
\usepackage{times}

\begin{document}

\sloppy

\spnewtheorem*{OortConj}{Oort Conjecture}{\bf}{\it}
\spnewtheorem*{StrongOC}{Ring Specific Local Oort Conjecture}{\bf}{\it}
\spnewtheorem*{LGP}{Local-to-Global Principle}{\bf}{\it}
\spnewtheorem*{LocGR}{Local Criterion for Good Reduction}{\bf}{\it}
\spnewtheorem*{BiratLL}{Birational Criterion for Local Lifting}{\bf}{\it}
\spnewtheorem*{Weier}{Weierstrass Preparation Theorem}{\bf}{\it}
\spnewtheorem*{ArithOort}{Arithmetic Form of the Ring Specific Local Oort Conjecture}{\bf}{\it}
\spnewtheorem*{ArithOortFp}{Arithmetic Form of the Ring Specific Local Oort Conjecture}{\bf}{\it}

\spnewtheorem{thm}{Theorem}{\bf}{\it}
\spnewtheorem{cor}[thm]{Corollary}{\bf}{\it}
\spnewtheorem{lem}[thm]{Lemma}{\bf}{\it}
\spnewtheorem{prop}[thm]{Proposition}{\bf}{\it}
\spnewtheorem{defin}[thm]{Definition}{\bf}{\it}
\spnewtheorem{rem}[thm]{Remark}{\bf}{\rm}
\spnewtheorem{examp}[thm]{Example}{\bf}{\it}

\numberwithin{equation}{section}
\numberwithin{thm}{section}

\newcommand{\f}{\mathfrak}
\newcommand{\mb}{\mathbb}
\newcommand{\mr}{\mathrm}
\newcommand{\mf}{\mathbf}
\newcommand{\mc}{\mathcal}
\newcommand{\e}{\emph}
\newcommand{\vp}{\varphi}
\newcommand{\Diff}{\textrm{Diff}}
\newcommand{\Norm}{\textrm{Norm}}
\newcommand{\Hom}{\textrm{Hom}}
\newcommand{\ch}{\textrm{char}}
\newcommand{\lcm}{\textrm{lcm}}
\newcommand{\wt}{\widetilde}
\title{Galois covers of the open $p$-adic disc}
%
\author{Scott Corry}
\institute{Lawrence University, Department of Mathematics, 711 E. Boldt Way, Appleton, WI 54911, USA. \email{corrys@lawrence.edu}}
\date{Received: date / Revised version: date}
%
%

\noindent
Scott Corry\\
\\
\noindent
{\large \textbf{Galois covers of the open $p$-adic disc}}
\\

\begin{abstract}
This paper investigates Galois branched covers of the open $p$-adic disc and their reductions to characteristic $p$. Using the field of norms functor of Fontaine and Wintenberger, we show that the special fiber of a Galois cover is determined by arithmetic and geometric properties of the generic fiber and its characteristic zero specializations. As applications, we derive a criterion for good reduction in the abelian case, and give an arithmetic reformulation of the local Oort Conjecture concerning the liftability of cyclic covers of germs of curves.
\end{abstract}

\section{Introduction}
\label{sec:introduction}

Let $R$ be a complete discrete valuation ring of mixed characteristic $(0,p)$, with perfect residue field $k$ and fraction field $K$. This paper concerns the reduction of Galois covers of the open $p$-adic disc $D:=\textrm{Spec}(R[[Z]])$. In particular, in section \ref{sec:ArithmeticOort} we use our main result (Theorem \ref{Main}) to give a characterization of the abelian covers which have good reduction to characteristic~$p$ (Proposition \ref{arithcrit}).

In order to describe our main result, let $Y\rightarrow D$ be a regular $G$-Galois branched cover, with $Y$ normal and with reduced special fiber $Y_k$. Then Theorem \ref{Main} provides a characterization of the special fiber $Y_k\rightarrow D_k$ in terms of the generic fiber $Y_K\rightarrow D_K$ and its characteristic zero specializations. The connection between the various fibers is effected by means of Wintenberger's field of norms functor \cite{W}, which we briefly recall in section \ref{sec:FieldofNorms}. Roughly speaking, our result says that the special fiber $Y_k\rightarrow D_k$ ``wants'' to be the field of norms of the characteristic zero fibers, and the failure of this identification is due to the presence of inseparability in the special fiber.

In order to relate the field of norms to the open $p$-adic disc, we choose a Lubin-Tate extension $L|K$, and consider the associated field of norms $X_K(L)$, a local field of characteristic $p$ with residue field $k$. The choice of a uniformizer for $X_K(L)$ yields an isomorphism $k((z))\xrightarrow[]{\sim} X_K(L)$, allowing us to identify $D_k=\textrm{Spec}(k[[z]])$ with the spectrum of the ring of integers $R_{X_K(L)}\subset X_K(L)$. Hence, we may view the special fiber $Y_k\rightarrow D_k$ as corresponding to a ring extension of $R_{X_K(L)}$.

On the other hand, the chosen uniformizer for $X_K(L)$ is a coherent system of norms, the components of which define a net of points $\{x^E\}\subset D_K$, and we may consider the collection of characteristic zero fibers $Y_E\rightarrow x^E$. Theorem \ref{Main} says that in the abelian case, the irreducibility of the fibers $Y_E$ implies the irreducibility of the special fiber $Y_k$. Moreover, for arbitrary groups $G$, if $Y_k$ is irreducible, then the separability of the special fiber is determined by the limiting behavior of the differents $d_E$ of the fibers $Y_E\rightarrow x^E$. Finally, when the special fiber $Y_k\rightarrow D_k$ is separable (but perhaps reducible), its generic fiber $Y_{k,\eta}\rightarrow D_{k,\eta}$ is obtained by applying the field of norms functor to the ``limit'' of the fibers $Y_E\rightarrow x^E$. 

The original motivation for this study was provided by the global lifting problem for Galois covers of curves. The general question is as follows: if $k$ is a algebraically closed field of characteristic $p>0$, and $f:C\rightarrow C'$ is a finite $G$-Galois branched cover of smooth projective $k$-curves, does there exists a lifting of $f$ to mixed characteristic? Via a local-global principle (\cite{GM} section III, \cite{BM} Corollaire 3.3.5), the global lifting problem reduces to the local lifting problem: given a finite $G$-Galois extension of power series rings
$k[[t]]|k[[z]]$, does there exist a lifting to a $G$-Galois extension $R[[T]]|R[[Z]]$, where $R$ is a mixed characteristic DVR with residue field $k$?  That is, the lifting problem becomes a question about Galois covers of the open $p$-adic disc. For an overview of the local and global lifting problems, see e.g. \cite{OSS}, \cite{GM}, \cite{Ga}, \cite{Gr}, \cite{WB}, \cite{CGH}.

The guiding conjecture in the subject was provided by F. Oort 
who suggested that cyclic covers should always lift (\cite{O}, I.7).
In section \ref{sec:ArithmeticOort}, we use our result to derive an arithmetic
reformulation of a stronger form of this conjecture, which specifies the ring
$R$ over which the lifting should occur:

\noindent
\begin{StrongOC}\footnote{We had originally intended to use the name Strong (Local) Oort Conjecture for this statement, but that name has recently been used in \cite{CGH} for a different strengthening (and generalization) of the Oort Conjecture.}
If $Y_k\cong \textrm{Spec}(k[[t]])$ and $f:Y_k\rightarrow D_k$ is cyclic of order $n$,
then $f$ lifts to a cover $Y\rightarrow D$ over $R=W(k)[\zeta_n]$, where $W(k)$
denotes the Witt vectors of $k$. (Here $k$ is algebraically closed of characteristic $p>0$.)
\end{StrongOC}

In section \ref{sec:FieldofNorms} of this paper we describe the theory of the field of norms due to Fontaine and
Wintenberger. 
Section \ref{sec:MT} contains the proof of Theorem \ref{Main} characterizing the special fiber 
of a Galois
branched cover of the open $p$-adic disc in terms of the characteristic zero
fibers of the cover. A criterion for good reduction in the abelian case is derived in section \ref{sec:ArithmeticOort}, together with an arithmetic reformulation of the Ring Specific Local Oort Conjecture. 

\subsection{Notation}
\label{sec:notation}

Let $K$ be a mixed characteristic complete discretely valued field with residue characteristic $p>0$. We
make the following notational conventions:
\begin{itemize}
\item[-] $R_K$ denotes the valuation ring of $K$;
\item[-] $\mathfrak{m}_K$ denotes the maximal ideal of $R_K$;
\item[-] $k_K$ denotes the residue field of $R_K$;
\item[-] $\nu_K$ denotes the normalized discrete valuation on $K$, so that 
$\nu_K(K^\times)=\mathbb{Z}$;
\item[-] $|\cdot|_K$ is the absolute value on $K$ induced by $\nu_K$, normalized
so that $|\alpha|_K=~p^{-\nu_K(\alpha)}$;
\item[-] if $L$ is the completion of an algebraic extension of $K$, then we
also denote by $\nu_K$ (resp. $|\cdot|_K$) the unique prolongation of $\nu_K$ 
(resp. $|\cdot|_K$) to $L$.
\end{itemize}

\section{The Field of Norms}
\label{sec:FieldofNorms}

\subsection{Arithmetically profinite extensions}
\label{sec:APF}

The field of norms construction applies to a certain type of field extension,
which we now describe. The basic reference for this material is \cite{W}.

\begin{defin}
Let $K$ be a complete discrete valuation field with perfect residue field $k_K$
of characteristic $p>0$, 
and $K^{sep}$ a fixed separable closure. Then an 
extension $L|K$ contained in $K^{sep}|K$ is called
\emph{arithmetically profinite (APF)} if for all $u\ge -1$, the group 
$G_K^uG_L$
is open in $G_K$.
\end{defin}
If we set $K^u:=\textrm{Fix}(G_K^u)\subset K^{sep}$, then this
definition means simply that $L^u := K^u\cap L$ is a finite extension of $K$
for all $u\ge -1$. 
Since the upper ramification filtration is separated, it follows that
$K^{sep}=\cup_u K^u$, which implies that ${L=\cup_u L^u}$. The fact that the ramification subextensions $L^u|K$ are finite and exhaust the extension $L|K$ is exactly the condition that allows one to define an inverse Herbrand function $\psi_{L|K}$ (see \cite{W} 1.2.1). 

\begin{examp}
The extension $\mathbb{Q}_p(\zeta_{p^\infty})|\mathbb{Q}_p$ is arithmetically profinite, and in this case
$L^m=\mathbb{Q}_p(\zeta_{p^m})$ for all  $m\ge 0$.
\end{examp}

An important quantity attached to an APF extension
$L|K$ is 
$$
i(L|K):=\textrm{sup}\{u\ge -1 \ | \ G_K^uG_L=G_K\}.
$$
In terms of the ramification subextensions $L^u|K$, the quantity $i(L|K)$ is the
supremum of the indices $u$ such that $L^u=K$. In the case where $L|K$ is Galois
with group $G=G_K/G_L$, we have $G^u=G_K^uG_L/G_L$, and $i(L|K)$ is the first
jump in the upper ramification filtration on $G$. 
Note that
$i(L|K)\ge 0$ if and only if $L|K$ is totally ramified, and $i(L|K)>0$ if and
only if $L|K$ is totally wildly ramified. 

Given an infinite APF extension $L|K$, let $\mathcal{E}_{L|K}$  denote the
set of finite subextensions of $L|K$, partially ordered by
inclusion. The key technical fact about the extension $L|K$ is the following
property of the quantity $i(-)$:

\begin{prop}\emph{(\cite{W}, Lemme 2.2.3.1)}\label{APFinfty} The numbers $i(L|E)$ for
$E\in\mathcal{E}_{L|K}$ tend to $\infty$ with respect to the directed set
$\mathcal{E}_{L|K}$.
\end{prop}

\subsubsection{Lubin-Tate extensions}\label{Coleman}

A special class of infinite APF extensions are 
the Lubin-Tate extensions, which we now briefly recall (see \cite{LT} and \cite{N} Chapter V). Let $H$ be a finite extension of 
$\mathbb{Q}_p$, and $\Gamma$ a Lubin-Tate formal group associated to a
uniformizer $\varpi$ of $H$. Then $\Gamma$ is a formal $R_H$-module, 
and interpreting the group $\Gamma$ in
$\mathfrak{m}^{sep}$ makes this ideal into an $R_H$-module (here
$\mathfrak{m}^{sep}$ is
the maximal ideal of the valuation ring of $H^{sep}$). Let
$\Gamma_m\subset\mathfrak{m}^{sep}$ be the $\varpi^m$-torsion of this
$R_H$-module. Then $R_H/\varpi^m R_H\cong \Gamma_m$ for all $m$, and in particular
$\Gamma_m$ is finite. Now define $L_0:=\cup_mH(\Gamma_m)$, which is an infinite 
totally
ramified abelian extension of $H$, the \emph{Lubin-Tate
extension of $H$ associated to $\varpi$}. Let $K$ be a complete unramified
extension of $H$ with Frobenius element $\phi\in \textrm{Gal}(K|H)$,
and define $L:=L_0K$. Then
$L|K$ is an infinite abelian APF
extension, with ramification subfields
$L^m:=\textrm{Fix}(G(L|K)^m)=K(\Gamma_m)$ (\cite{N}, Corollary V.5.6).
Moreover, we have
$[L^m:K]=q^{m-1}(q-1)$, where $q:=\#(k_H)$.
We will refer to extensions of this
type as Lubin-Tate extensions, despite the fact that they are really the
compositum of a Lubin-Tate extension with an unramified extension.

\subsection{The field of norms}
\label{sec:FON}

Having introduced infinite APF extensions, we are now
ready to describe the field of norms construction, following \cite{W}: given an infinite APF
extension $L|K$, set
$$
X_K(L)^* = \lim_{\xleftarrow[\mathcal{E}_{L|K}]{}}E^*,
$$
the transition maps being
given by the norm $N_{E'|E}:E'^*\rightarrow E^*$ for $E\subset E'$. Then define
$X_K(L)=X_K(L)^*\cup\{0\}.$
Thus, a nonzero element $\alpha$ of $X_K(L)$ is given by a norm-compatible sequence
$\alpha=(\alpha_E)_{E\in\mathcal{E}_{L|K}}$.
We wish to endow this set with an additive structure in such a way that $X_K(L)$
becomes a field, called the \emph{field of norms} of $L|K$. This is
accomplished by the following
\begin{prop}\emph{(\cite{W}, Th\'eor\`eme 2.1.3 (i))} If $\alpha,\beta\in
X_K(L)$, then for all ${E\in\mathcal{E}_{L|K}}$, the elements
$\{N_{E'|E}(\alpha_{E'}+\beta_{E'})\}_{E'}$ converge (with respect to the directed set
$\mathcal{E}_{L|E}$) to an element $\gamma_E\in E$. Moreover, $\alpha+\beta :=
(\gamma_E)_{E\in\mathcal{E}_{L|K}}$ is an element of $X_K(L)$.
\end{prop}
With this definition of addition, the set $X_K(L)$ becomes a field, with
multiplicative group $X_K(L)^*$. Moreover, there is a natural discrete
valuation on $X_K(L)$. Indeed, if $L^0$ denotes the maximal unramified subextension of 
$L|K$ (which is finite over K by APF), then 
$\nu_{X_K(L)}(\alpha):=\nu_E(\alpha_E)\in\mathbb{Z}$
does not depend on $E\in\mathcal{E}_{L|L^0}$. In fact 
(\cite{W}, Th\'eor\`eme 2.1.3 (ii)),  $X_K(L)$ is a complete
discrete valuation
field with residue field isomorphic to $k_L$ 
(which is a finite extension of $k_K$). The isomorphism of residue fields
$k_{X_K(L)}\cong k_L$ comes about as follows. For $x\in k_L$, let $[x]\in L^0$
denote the Teichm\"uller lifting. Note that $E|L^1$ is of  
$p$-power degree for all
$E\in\mathcal{E}_{L|L^1}$, so $x^\frac{1}{[E:L^1]}\in k_L$ for all such $E$,
since $k_L$ is perfect.
The element $([x^\frac{1}{[E:L^1]}])_{E\in\mathcal{E}_{L|L^1}}$ is clearly a
coherent system of norms, hence (by cofinality) 
defines an element $f_{L|K}(x)\in X_K(L)$. The map $f_{L|K}:k_L\rightarrow
X_K(L)$ is a field embedding which induces the isomorphism $k_L \cong
k_{X_K(L)}$ mentioned above.

The following result will be used several times in the proof of our main result,
Theorem~\ref{Main}. Before stating it, we make a
\begin{defin} For any subfield $E\in \mathcal{E}_{L|K}$, define
$r(E):=\left\lceil\frac{p-1}{p}i(L|E)\right\rceil.$
\end{defin}

\begin{prop}\emph{(\cite{W}, Proposition 2.3.1 \& Remarque 2.3.3.1)}\label{ApproxAPF}
Let $L|K$ be an infinite APF extension and $F\in\mathcal{E}_{L|L^1}$ be any
finite extension of $L^1$ contained in $L$. Then
\begin{enumerate}
\item for any $x\in R_F$, there exists 
$\hat{x}=(\hat{x}_E)_{E\in\mathcal{E}_{L|K}}\in X_K(L)$ such that
$$
\nu_F(\hat{x}_F-x)\ge r(F);
$$
\item for any $\alpha, \beta\in R_{X_K(L)}$, we have
$$
(\alpha+\beta)_F\equiv \alpha_F+\beta_F \mod \mathfrak{m}_F^{r(F)}.
$$
\end{enumerate}
\end{prop}
This proposition says that an element of $R_F$ can be approximated by an element of the field of norms $X_K(L)$, and that the addition in $X_K(L)$ approximates the addition in $R_F$.  The error in these approximations has valuation at least $r(F)$ in the field $F$.

The construction just described, which produces a complete discrete valuation
field of characteristic
$p=\textrm{char}(k_K)$ from an infinite APF extension $L|K$ is actually functorial
in $L$ (see \cite{W} 3.1). Precisely, $X_K(-)$ can be viewed as a functor from the category of
infinite APF extensions of $K$ contained in $K^{sep}$ (where the morphisms are
$K$-embeddings of finite degree) to the category of complete discretely valued
fields of
characteristic $p$ (where the morphisms are separable embeddings
of finite degree). Moreover, this functor preserves Galois extensions and
Galois groups. 

Fixing an infinite APF extension $L|K$, the functorial nature of $X_K(-)$
allows us to define a field of norms for \emph{any} separable algebraic 
extension $M|L$. Namely, given such an $M$, define the directed set
$\mathcal{M}:=\{L'\subset M \ |\ [L': L]<\infty\}$, and note that
$$
M=\lim_{\xrightarrow[\mathcal{M}]{}}L'.
$$
Then we define the field of norms
$$
X_{L|K}(M) := \lim_{\xrightarrow[\mathcal{M}]{}}X_K(L').
$$
With this
definition, we can consider $X_{L|K}(-)$ as a functor from the category of
separable algebraic extensions of $L$ to the category of separable algebraic
extensions of $X_K(L)$. 
\begin{prop}\emph{(\cite{W}, Th\'eor\`eme 3.2.2)}\label{Cats} The field of norms functor 
$X_{L|K}(-)$ is an equivalence of categories.
\end{prop}
\noindent
In particular, $X_{L|K}(K^{sep})$
is a separable closure of $X_K(L)$, and we have an isomorphism 
$G_{X_K(L)}\cong G_L$.

Since $X_K(L)$ is a complete discrete valuation field with residue field $k_L$,
it follows that any choice of uniformizer $\pi=(\pi_E)_E$ for $X_K(L)$ yields an
isomorphism $k_L((z))\cong X_K(L)$, defined by sending $z$ to $\pi$. Via this
isomorphism, an element $\alpha=(\alpha_E)_E\in R_{X_K(L)}$ 
corresponds to a power series $g_\alpha(z)\in k_L[[z]]$. The following lemma
describes the relationship between $g_\alpha(z)$ and the coherent system of norms
$\alpha=(\alpha_E)_E$ in terms of the chosen uniformizer $\pi=(\pi_E)_E$. First
we need to introduce some notation. Given a power series
$$
g(z)=\sum_{i=0}^\infty a_iz^i\in k_L[[z]],
$$
define for each $E\in\mathcal{E}_{L|L^1}$ a new power series
$$
g_E(z):=\sum_{i=0}^\infty [a_i^{\frac{1}{[E:L^1]}}]z^i=
\sum_{i=0}^\infty (f_{L|K}(a_i))_E z^i\in R_E[[z]].
$$

\begin{lem}\label{SpecialCong}
For all $\alpha=(\alpha_E)_E\in X_K(L)$, we have
$\alpha_E\equiv g_{\alpha,E}(\pi_E) \mod \mathfrak{m}_E^{r(E)}$
for all $E\in\mathcal{E}_{L|L^1}$.
\end{lem}

\begin{proof}
By definition of the isomorphism $k_L((z))\xrightarrow{\sim} X_K(L)$, if
$g_\alpha(z)=\sum_{i=0}^\infty a_iz^i$, then
$$
\alpha=\sum_{i=0}^\infty f_{L|K}(a_i)\pi^i=
\lim_{n\rightarrow\infty}\sum_{i=0}^n f_{L|K}(a_i)\pi^i.
$$
Now by Proposition~\ref{ApproxAPF}, for any $E\in\mathcal{E}_{L|L^1}$ we have 
$$
\left(\sum_{i=0}^n f_{L|K}(a_i)\pi^i\right)_E\equiv
\sum_{i=0}^n (f_{L|K}(a_i))_E\pi_E^i \mod \mathfrak{m}_E^{r(E)}.
$$
Thus we see that
$$
\alpha_E\equiv\lim_{n\rightarrow\infty}\sum_{i=0}^n (f_{L|K}(a_i))_E\pi_E^i
=g_{\alpha,E}(\pi_E) \mod \mathfrak{m}_E^{r(E)}.~\qed
$$
\end{proof}

\subsection{Connection with the open $p$-adic disc}\label{Disc}

Given a totally ramified infinite APF extension $L|K$, we have seen how any
choice of a uniformizer $\pi=(\pi_E)_E\in X_K(L)$ determines an isomorphism
$k((z))\cong X_K(L)$ defined by sending $z$ to $\pi$ (here we set $k:=k_K=k_L$). We would now like to
explicitly describe a connection between the field of norms $X_K(L)$ and the
open $p$-adic disc $D_K:=\textrm{Spec}(R[[Z]]\otimes K)$ that will underly the
rest of our investigation (here $R:=R_K$). Namely, the special fiber of the smooth integral
model $D:=\textrm{Spec}(R[[Z]])$ is $D_k=\textrm{Spec}(k[[z]])$, with generic
point $D_{k,\eta}=\textrm{Spec}(k((z)))$. Via the isomorphism above coming from
the choice of uniformizer $\pi$, we can thus identify $D_{k,\eta}$ with
$\textrm{Spec}(X_K(L))$. On the other hand, each component $\pi_E$ of $\pi$ is a
uniformizer in $E$, and in particular has absolute value $|\pi_E|_K<1$.
Hence, each $\pi_E$ corresponds to a point $x^E\in D_K$ with residue field $E$.
In terms of the Dedekind domain $R[[Z]]\otimes K$, the point $x^E$ corresponds
to the maximal ideal $\mathcal{P}_E$ generated by the minimal polynomial of
$\pi_E$ over $R$. Thus, the uniformizer $\pi$ defines a net of points
$\{x^E\}_E\subset D_K$ which approaches the boundary. In summary, we have the following picture:
\newline

$\begin{CD}
D_{k,\eta} \qquad\qquad@= \qquad \textrm{Spec}(X_K(L)) \qquad\qquad\qquad\qquad D_K\\
k((z)) \qquad\qquad @>>> \qquad\quad\qquad X_K(L) \qquad\qquad\qquad\quad R[[Z]]\otimes K\\
 \ z  \qquad\qquad @>>> \ \qquad\qquad\pi=(\pi_E)_E \qquad \leftrightsquigarrow \qquad \quad \{x^E\}_E
\end{CD}$

\section{The Main Theorem}
\label{sec:MT}

Let $L|K$ be a Lubin-Tate extension as described in section~\ref{Coleman}, with residue field
$k:=k_K=k_L$. Hence, there exists a $p$-adic local field $H$ such that $K|H$ is
unramified and $L=KL_0$, where $L_0|H$ is an honest Lubin-Tate extension,
associated to a formal group $\Gamma$. As usual, we let 
$L^m:=\textrm{Fix}(G(L|K)^m)$, and we recall that
$[L^m:L^1]=\#(k_H)^{m-1}=q^{m-1}.$ Choose a uniformizer $\pi=(\pi_E)_E\in X_K(L)$, which yields the identification
$D_{k,\eta}=\textrm{Spec}(X_K(L))$ as well as the net of points 
$\{x^E\}_E\subset D_K$ as
described in the last section. 

Let $G$ be a finite group, and consider a $G$-Galois regular branched cover $Y\rightarrow D$, with $Y$ normal. We consider this cover to be a family over
$\textrm{Spec}(R_K)$, and we introduce the following notations:
\begin{itemize}
\item[-] $Y_k\rightarrow D_k$ denotes the special fiber of the cover, obtained by taking the
fibered product with $\textrm{Spec}(k)$;
\item[-] $Y_K\rightarrow D_K$ denotes the generic fiber, obtained by
taking the fibered product with $\textrm{Spec}(K)$;
\item[-] for each $E\in\mathcal{E}_{L|K}$, we denote by $Y_E$ the fiber
of $Y_K$ at $x^E\in D_K$;
\item[-] if $X$ is an affine scheme,
then $F(X)$ denotes the total ring of fractions of $X$, obtained from the ring
of global sections, $\Gamma(X)$, by inverting all non-zero-divisors;
\item[-] if the special fiber $Y_k$ is reduced, then $F(Y_k)\cong \prod_{j=1}^{n_s}\mathcal{K}$ is a product of 
$n_s$ copies of a field $\mathcal{K}$,
which is a finite normal extension of $k((z))=X_K(L)$;
\item[-]since only finitely many of the points $x^E$ are ramified in the cover
${Y_K\rightarrow D_K}$, for $E$ large the fiber $Y_E$ is reduced and we have
an isomorphism ${F(Y_E)\cong\prod_{j=1}^{n_E} E'}$,
where $E'|E$ is a finite Galois extension of fields;
\item[-] $d_E:=\nu_{E'}(\mathcal{D}(E'|E))$ denotes the degree of
the different of $E'|E$;
\item[-] $L_E:=LE'$ denotes the compositum of the fields $L$ and $E'$ in $K^{sep}$.
\end{itemize} 

\begin{thm}\label{Main} Let $Y\rightarrow D$ be a $G$-Galois regular branched cover of the open $p$-adic
disc, with $Y$ normal and $Y_k$ reduced. Then
\begin{enumerate}

\item If $Y_k\rightarrow D_k$ is generically separable, then there exists a cofinal set $\mathcal{C}_Y\subset\mathcal{E}_{L|K}$ 
such that for $E\in\mathcal{C}_Y$ large, we have $n_E=n_s$ and
$\mathcal{K}=X_{L|K}(L_E)$
as subfields of $X_K(L)^\textrm{sep}=X_{L|K}(K^\textrm{sep})$.
Moreover, for these $E$, the functor $X_{L|K}(-)$ induces an isomorphism
$\textrm{Gal}(\mathcal{K}|X_K(L))\cong \textrm{Gal}(E'|E)$
which respects the ramification filtrations. In particular, if $d_s$ is the
degree of the different of $\mathcal{K}|X_K(L)$, then $d_s=d_E$.

\item If $Y_k$ is irreducible, then $Y_k\rightarrow D_k$ is generically
inseparable if and only if $d_E\rightarrow\infty$.

\item If $G$ is abelian, then $n_s\le n_E$ for $E$ large, independently of any separability assumption. In particular, $Y_k$ is irreducible if $Y_E$ is irreducible for $E$ large.
\end{enumerate}

\end{thm}

\begin{rem}
If $k$ is a finite field, say $\#(k)=q^t$, then we can take the cofinal set $\mathcal{C}_Y$ in part 1 of
the Theorem to be $\{L^m \ | \ m\equiv 1 \mod t\}$. That is, in the case of a finite residue field, 
$\mathcal{C}_Y$ is
independent of the particular cover $Y\rightarrow D$.
\end{rem}

Before beginning the proof, we first describe two simple arguments that will be used repeatedly.

\subsection{The Weierstrass Argument}\label{WArg}

As a consequence of the Weierstrass Preparation Theorem (\cite{Bour} VII.3.8, Prop. 6), an
arbitrary nonzero element $A(Z)\in R[[Z]]_{\mathfrak{m}}$ has
the form
$$
A(Z)=\frac{\varpi^cf_1(Z)U(Z)}{f_2(Z)},
$$
where the $f_i(Z)$ are distinguished polynomials, $U(Z)$ is a unit in $R[[Z]]$, $\varpi$ is a uniformizer for $R$, and $c\ge 0$. In particular, the denominator $f_2(Z)$
will be relatively prime to almost all height one primes of $R[[Z]]$, so if
$\mathcal{P}=(h(Z))$ is one of these primes, we will have $A(Z)\in R[[Z]]_\mathcal{P}$,
and it will make sense to look at the image of $A(Z)$ in
$R[[Z]]_\mathcal{P}/\mathcal{P}\cong K(\alpha)$, where $\alpha$ is a root of
$h(Z)$ in $K^{sep}$. When we have chosen a particular root $\alpha$, we
will refer to the image of $A(Z)$ in $K(\alpha)$ as the \emph{specialization} of $A$ at the
point $Z=\alpha$, and denote it by $A(\alpha)$. More generally, if 
\begin{equation}\label{Spoly}
S(T)=T^N+A_{N-1}(Z)T^{N-1}+\cdots +A_0(Z)
\end{equation}
is a polynomial with coefficients in $R[[Z]]_\mathfrak{m}$, then we can apply the previous reasoning to each of the finitely many coefficients $A_i(Z)$. We conclude that for almost all points $Z=\alpha$, we can specialize to obtain the polynomial
$$
S(T)|_{Z=\alpha}=T^n+A_{N-1}(\alpha)T^{N-1}+\cdots+A_0(\alpha)\in K(\alpha)[T].
$$
In what follows, we
will refer to this argument (which allows us to specialize polynomials almost everywhere)
as the \emph{Weierstrass Argument}.

\subsection{The Ramification Argument}\label{RArg}

Suppose that $\{x_m\}_m\subset D_K$ is a sequence of points corresponding to a
sequence $\{\alpha_m\}_m\in K^{sep}$ with each $\alpha_m$ being a
uniformizer for the discrete valuation field $K(\alpha_m)$.
Moreover, suppose that $|\alpha_m|_K\rightarrow 1$ as $m\rightarrow \infty$, so
that the points $x_m$ are approaching the boundary of $D_K$. Equivalently, we
are assuming that the ramification index $e_m:=e(K(\alpha_m)|K)$ goes to $\infty$
with $m$. Given 
$A(Z)=\varpi^c\frac{f_1(Z)}{f_2(Z)}U(Z)\in R[[Z]]_\mathfrak{m}$, we can
consider the specialization of $A$ at $Z=\alpha_m$ for $m>>0$ (by the Weierstrass
Argument). We find that
$$
\nu_{K(\alpha_m)}(A(\alpha_m))=c\nu_{K(\alpha_m)}(\varpi)+
\nu_{K(\alpha_m)}\left(\frac{f_1(\alpha_m)}{f_2(\alpha_m)}\right).
$$
Letting $d_i=\deg(f_i)$, observe that for any $a\in\mathfrak{m}_K$ we have 
$\nu_{K(\alpha_m)}(a)\ge\nu_{K(\alpha_m)}(\varpi)=e_m\ge \max\{d_1,d_2\}$ for $m>>0$. It follows
that ${\nu_{K(\alpha_m)}(f_i(\alpha_m))=\nu_{K(\alpha_m)}(\alpha_m^{d_i})=d_i}$, so
${\nu_{K(\alpha_m)}(A(\alpha_m))=ce_m+(d_1-d_2)\ge~d_1-d_2.}$
Thus, we see that the normalized valuations of the specializations $A(\alpha_m)\in
K(\alpha_m)$ are bounded below, independently of $m$. Moreover, if $c>0$ (i.e. if $\overline{A}(z)=0$), then 
$\nu_{K(\alpha_m)}(A(\alpha_m))\rightarrow\infty$ as $m\rightarrow\infty$, and
if $d_1\ge d_2$, then $A(\alpha_m)\in R_{K(\alpha_m)}$ for $m>>0$, even if $c=0$.

Applying the preceding remarks to the finitely many coefficients of a polynomial $S(T)\in R[[Z]]_{\mathfrak{m}}$ as in (\ref{Spoly}), we obtain a uniform lower bound on the normalized valuations of the coefficients of $S(T)|_{Z=\alpha_m}$, 
independently of $m$. As described above, it is easy to check whether
these specialized coefficients are integral, and whether their valuations remain bounded as $m\rightarrow\infty$. We will refer to this argument (which yields information on the valuations of specializations) as the \emph{Ramification Argument} in the sequel.\\

\noindent
\emph{Proof of Theorem~\ref{Main}.} We begin by translating the geometric hypotheses of Theorem~\ref{Main} into algebraic statements.
Let $Y=\textrm{Spec}(\mathcal{A})$, so that
$\mathcal{A}|R[[Z]]$ is a $G$-Galois extension of normal rings (here $R=R_K$). The hypothesis that $Y_k$ is reduced means that
$\mathcal{A}_s:=\mathcal{A}/\varpi\mathcal{A}$
is reduced, where $\varpi$ is a uniformizer of $R$. Moreover, 
$F(Y_E)=(\mathcal{A}\otimes
K)/\mathcal{P}_E(\mathcal{A}\otimes K)=\prod_{j=1}^{n_E} E'$ for $E$ large (here $\mathcal{P}_E$ is
the maximal ideal of $R[[Z]]\otimes K$ corresponding to the point $x^E\in D_K$).

Since the proof of part 1 is long and technical, we present below a brief outline describing the strategy:
\begin{itemize}
\item[A)] Start by finding a polynomial $f(T)\in k[z][T]$ such that $F(Y_k)\cong k((z))[T]/(f(T))$.
\item[B)] Then take a suitable lifting $F(T)\in R[[Z]]_{(\varpi)}[T]$ of $f(T)$, and for each $E\in\mathcal{E}_{L|K}$, consider the specialized polynomial $F_E(T):=F(T)|_{Z=\pi_E}.$ The polynomial $F(T)$ has the property that if $y^E$ is a root of $F_E(T)$ in $K^{sep}$, then $E'=E(y^E)$ for $E$ sufficiently large.
\item[C)] Using the identification $X_K(L)=k((z))$, prove the following equality of discriminants: $\nu_{X_K(L)}(\textrm{disc}(f))=\nu_E(\textrm{disc}(F_E))$ for $E$ sufficiently large.
\item[D)] For each $E\in\mathcal{E}_{L|K}$, approximate $y^E\in E'$ by an element $\hat{y}^E\in X_K(L_E)$.
\item[E)] Show that a subnet of the net $\{\hat{y}^E\}_E$ converges to a root of $f$ in $X_K(L)^{sep}$, and that this root generates the field extension $\mathcal{K}|X_K(L)$.
\item[F)] The preceding steps combined with Krasner's Lemma allow us to conclude that $\mathcal{K}=X_L(L_E)$ for $E$ sufficiently large, and the statement about Galois groups follows from the properties of the field of norms.
\end{itemize}

\begin{rem}
The knowledgeable reader will note that the strategy above is
inspired by the proof in \cite{W} of the essential surjectivity
statement in Th\'eor\`eme 3.2.2 (reproduced above as Proposition \ref{Cats}). The main difficulty is to spread the
construction of \cite{W} over the open $p$-adic disc.
\end{rem}

\noindent
{\it Proof of part 1}\\

\noindent
A) Suppose that $Y_k\rightarrow D_k$ is generically separable, which means that
the field extension $\mathcal{K}|k((z))$ is separable, hence Galois. By the Primitive Element Theorem, there exists $\chi\in \mathcal{K}$ such that
$\mathcal{K}=k((z))[\chi]$. Moreover, we can choose $\chi$ to be integral over
$k[[z]]$, say with
minimal polynomial $f(T)\in k[[z]][T]$. Further, since $k((z))$ is infinite, we
can choose $n_s$ different primitive elements $\chi_j\in\mathcal{K}$ such that the
corresponding minimal polynomials $f_j(T)\in k[[z]][T]$ are distinct.
Even more, by Krasner's Lemma, we may assume
that each $f_j(T)\in k[z][T]$, so that in fact $f_j(T)\in \mathbb{F}_{q^l}[T]$
for some $l>0$. 
Having fixed this $l$, we take ${\mathcal{C}_Y = \{L^m \ | \ m\equiv 1 \mod l\}}$.

Setting $f(T):=\prod_{j=1}^{n_s}f_j(T)$, the Chinese Remainder Theorem 
implies that we have an isomorphism 
$$
k((z))[T]/(f(T))\cong\prod_{j=1}^{n_s}k((z))[\chi_j]=
\prod_{j=1}^{n_s}\mathcal{K}\cong F(Y_k)
$$

\noindent
B) Let $\overline{\xi}$ be the element of $F(Y_k)$ corresponding to $\overline{T}$ under this isomorphism,
and
choose a lifting, $\xi$, of $\overline{\xi}$ to
$\mathcal{A}_{(\varpi)}$. Denote the minimal
polynomial of $\xi$ over $R[[Z]]_{(\varpi)}$ by $F(T)$, so that $\overline{F}(T)=f(T)$.
Then $F(T)$ has the form
$$
F(T)=T^N+A_{N-1}(Z)T^{N-1}+\cdots +A_0(Z)\in R[[Z]]_{(\varpi)}[T].
$$
Now by the Weierstrass Preparation Theorem, the coefficients of
$F(T)$ have the form
$$
A_i(Z)=\frac{g(Z)}{Z^n+a_{n-1}Z^{n-1}+\cdots+a_0},
$$
where $g(Z)\in R[[Z]]$ and the denominator is a distinguished polynomial.
Moreover, because $\overline{F}(T)=f(T)\in k[[z]][T]$, it follows that each
$\overline{A_i}(z)\in k[[z]]$, which implies that either $\varpi | g(Z)$ 
in $R[[Z]]$ (in which case
$\overline{A_i}(z)=0$), or the Weierstrass degree of $g(Z)$ is
greater than $n$ (the degree of the denominator).

Now by the Weierstrass Argument described in section \ref{WArg}, for $E$ large we can specialize the polynomial
$F(T)$ at the point $Z=\pi_E$ to obtain the polynomial $F_E(T)\in E[T]$. 
\begin{lem}
For $E$ large, the specialized polynomial $F_E(T)$ lies in $R_E[T]$, where $R_E$ is the
valuation ring of $E$.
\end{lem} 
\begin{proof}
This follows immediately from the previous remarks and the
Ramification Argument applied to $S(T)=F(T)$ in the notation of section \ref{RArg}.~$\qed$
\end{proof}

Now let $y^E$ be a root of $F_E(T)$ in $K^{\textrm{sep}}$, and consider the field extension
$E(y^E)|E$. 

\begin{lem}
For $E$ large we have $E(y^E)=E'$, where $F(Y_E)\cong\prod_{j=1}^{n_E}E'$. In particular, $L_E:=LE'=L(y^E)$, and $L(y^E)$ is Galois over $L$.
\end{lem}
\begin{proof}
Take $g$ to be the product of the
denominators of the coefficients $A_i(Z)$ of $F(T)\in R[[Z]]_{(\varpi)}[T]$.
Then the conductor of the subring 
$(R[[Z]]\otimes K)_g[\xi]\subset(\mathcal{A}\otimes K)_g$ 
defines a closed subset of $Y_K$, and if
$x^E$ lies outside the image of this set in $D_K$, then the splitting of $F(T)$
mod $\mathcal{P}_E$ determines the fiber $Y_E$ (see \cite{N},
Prop. I.8.3).~$\qed$
\end{proof}

\noindent
C) Since $f(T)\in k[[z]][T]$ is separable, we have
$\textrm{disc}(f) 
\ne 0$. Using the identification $k((z))=X_K(L)$, the discriminant of $f$ becomes a coherent system of norms: $\textrm{disc}(f)=(\textrm{disc}(f)_E)_E$. Since $r(E):=\lceil\frac{p-1}{p}i(L|E)\rceil$, we know by Proposition~\ref{APFinfty} that
$\lim_{E\in\mathcal{E}_{L|K}}r(E)=\infty$, so there exists $E_0$ such that for
$E\ge E_0$ we have 
\begin{eqnarray}\label{E_0}
r(E)\ge r(E_0)>\nu_{X_K(L)}(\textrm{disc}(f)):=\nu_E(\textrm{disc}(f)_E).
\end{eqnarray}
Now
$$
f(T)=\overline{F}(T)
=T^N+\overline{A_{N-1}}(z)T^{N-1}+\cdots+\overline{A_0}(z)\in
k[z][T].
$$
As above, each 
coefficient $\overline{A_i}(z)$ corresponds to a coherent system of norms
$\alpha_i=(\alpha_{i,E})_E$. Hence, we can write
$$
f(T)=T^N+\alpha_{N-1}T^{N-1}+\cdots+\alpha_0\in R_{X_K(L)}[T].
$$
Now let $f_E(T)\in R_E[T]$ be the polynomial obtained from $f(T)$ by selecting
the $E$th component from each coefficient:
$$
f_E(T):=T^N+\alpha_{{N-1},E}T^{N-1}+\cdots+\alpha_{0,E}\in R_E[T].
$$
For the convenience of the reader, we summarize the notation introduced so far:
\begin{itemize}
\item[-] $\textrm{disc}(f)_E\in E^*$ is the $E$th component of the discriminant of the polynomial $f(T)\in X_K(L)[T]$.
\item[-] $f_E(T)\in R_E[T]$ is the polynomial obtained from $f(T)$ by selecting the $E$th component of each coefficient. In particular, $\textrm{disc}(f_E)\in E$ refers to the discriminant of the polynomial $f_E$, and $\textrm{disc}(f)_E\ne\textrm{disc}(f_E)$ in general.
\item[-] $F_E(T)\in R_E[T]$ is the specialization of the polynomial $F(T)\in R[[Z]]_{(\varpi)}[T]$ at $Z=\pi_E$.
\item[-] $E_0\in\mathcal{E}_{L|K}$ has the property that $E\ge E_0$ implies that $r(E_0)>\nu_E(\textrm{disc}(f)_E)$.
\end{itemize}

\begin{lem}\label{disc}
For $E\in\mathcal{C}_Y$ large, we have
$\nu_{X_K(L)}(\emph{disc}(f))=\nu_{E}(\emph{disc}(F_E))$.
\end{lem}
\begin{proof}
Define $\tau:k[[z]] - {0}\rightarrow R[[Z]]$ to be the coefficient-wise Teichm\"uller lifting of power series:
$$
\tau(\sum_i a_iz^i):= \sum_i [a_i]Z^i.
$$
Then let
$G(T)\in R[Z][T]$ be the Teichm\"uller lifting of $f(T)$:
$$
G(T):=\tau(f)(T)=T^N+\tau(\overline{A_{N-1}})(Z)T^{N-1}+\cdots+\tau(\overline{A_0})(Z).
$$
Both $G$ and $F$
reduce mod $\varpi$ to $f$, hence 
$F(T)=G(T)+\varpi g(Z,T)$
for some $g(Z,T)\in R[[Z]]_{(\varpi)}[T]$. Specializing at $Z=\pi_E$ for $E\in\mathcal{C}_Y$ now yields the
equation
\begin{eqnarray}\label{specialeq}
F_E(T)=f_E(T)+\pi_E^{r(E)}h_E(T)+\varpi g(\pi_E,T)
\end{eqnarray}
for some $h_E(T)\in R_E[T]$. Indeed, by Lemma~\ref{SpecialCong}, we have
$\overline{A_i}_{,E}(\pi_E)\equiv\alpha_{i,E} \mod \mathfrak{m}_{E}^{r(E)}.$
But $[E:L^1]=q^{lt_E}=(q^l)^{t_E}$ for $E\in\mathcal{C}_Y$. The
operation of raising 
to the $q^l$th power on $\mathbb{F}_{q^l}$ is the identity, and since the
coefficients of $\overline{A_i}(z)$ lie in $\mathbb{F}_{q^l}$, it follows that
$\overline{A_{i}}_{,E}(Z)=\tau(\overline{A_i})(Z)$. Hence
$\tau(\overline{A_i})(\pi_E)\equiv\alpha_{i,E} \mod \mathfrak{m}_{E}^{r(E)}$,
from which equation~(\ref{specialeq}) follows immediately.

Note that the discriminant is given by a polynomial expression of the coefficients, so there is a polynomial $D\in\mathbb{Z}[x_0,\cdots,x_{N-1}]$ such that 
\begin{eqnarray*}
\textrm{disc}(f)&=&D(\alpha_0,\cdots,\alpha_{N-1})\in X_K(L) \qquad \textrm{and}\\
\textrm{disc}(f_E)&=&D(\alpha_{0,E},\cdots,\alpha_{N-1,E})\in E.
\end{eqnarray*}
But by Proposition~\ref{ApproxAPF}, we have that
$$
D(\alpha_0,\cdots,\alpha_{N-1})_E\equiv D(\alpha_{0,E},\cdots,\alpha_{N-1,E}) \quad \textrm{mod}
\ \ \mathfrak{m}_E^{r(E)}.
$$
This means that 
$\nu_{E}(\textrm{disc}(f)_{E}-\textrm{disc}(f_E))\ge r(E)$.  

Moreover, consideration of the Taylor expansion for $D(x_0,\dots,x_{N-1})$ at the point $(\alpha_{0,E},\dots,\alpha_{N-1,E})$ shows that for $E$ large we have
\begin{eqnarray*}
\nu_{E}(\textrm{disc}(f_E)-\textrm{disc}(F_E))&=&
\nu_{E}(\textrm{disc}(f_E)-\textrm{disc}(f_E+\pi_E^{r(E)}h_E+\varpi g(\pi_E,T)))\\
&\ge& r(E_0).
\end{eqnarray*}
Indeed, $r(E)\rightarrow \infty$, and the Ramification Argument (section \ref{RArg}) applied to $S(T)=\varpi g(Z,T)$ shows that the valuations of the coefficients of $\varpi g(\pi_E, T)$ also go to infinity.

Putting the previous two paragraphs together yields the inequality
\begin{eqnarray*}
\nu_{E}(\textrm{disc}(f)_{E}-\textrm{disc}(F_E))&=&\nu_{E}(\textrm{disc}(f)_{E}-\textrm{disc}(f_E)+
\textrm{disc}(f_E)-\textrm{disc}(F_E))\\
&\ge&\min\{r(E),r(E_0)\}=r(E_0).
\end{eqnarray*}
Since $\nu_E(\textrm{disc}(f)_E)<r(E_0)$ by (\ref{E_0}), we conclude that for $E\in\mathcal{C}_Y$ large we must have
$$
\nu_{X_K(L)}(\textrm{disc}(f)):=\nu_{E}(\textrm{disc}(f)_{E})=\nu_{E}(\textrm{disc}(F_E)). \ \ \qed
$$
\end{proof}

\noindent
D) We now introduce the
following lemma from \cite{W}, adapted to our context:

\begin{lem}\emph{(\cite{W}, Lemme 3.2.5.4)}\label{LinDis}
For $E\subset{C}_Y$ large, the extensions $E'|E$ and $L|E$ are linearly disjoint.
Moreover, we have
$$
i(L_E|E')=\psi_{E'|E}(i(L|E))\ge i(L|E).
$$
\end{lem}

\begin{proof}
The proof in \cite{W} is valid once the following notational identifications have been made: replace $E_n$ by $E$, $E_n'$ by $E'$, and $L_n'$ by  $L_E$. Also, replace Wintenberger's polynomials $f_n$ by our specialized polynomials $F_E$. The key ingredient of the proof is Lemma \ref{disc} proven above. $\qed$
\end{proof}

Since for $E\subset\mathcal{C}_Y$ large, $L|E$ is totally wildly ramified, it follows from this lemma that
$i(L_E|E')\ge i(L|E)> 0,$
so $L_E|E'$ is totally wildly ramified. Hence,
Proposition~\ref{ApproxAPF} says that there exists 
$\hat{y}^E=(\hat{y}^E_B)_B\in X_K(L_E)$
such that 
$\nu_{E'}(\hat{y}^E_{E'}-y^E)\ge r(E').$\\

\noindent
E) We wish to prove that $\{\hat{y}^E\}_E$ possesses a subnet converging to a root of $f$ in $X_K(L)^{sep}$.
\begin{lem}\label{limit}
$\lim_{E\in\mathcal{C}_Y}f(\hat{y}^E)=0$.
\end{lem}

\begin{proof}
Note that $[X_K(L_E):X_K(L)]\le\deg(F_E)=\deg(F)=\deg(f)$, so we have
$$
\nu_{X_K(L)}(f(\hat{y}^E))\ge \frac{1}{\deg(f)}\nu_{X_K(L_E)}(f(\hat{y}^E)).
$$
Moreover, as mentioned above, $L_E|E'$ is totally ramified, hence
$$
\nu_{X_K(L_E)}(f(\hat{y}^E))=\nu_{E'}(f(\hat{y}^E)_{E'}).
$$
Denote by $f_{E'}\in E'[T]$ the polynomial
obtained by replacing each coefficient of $f\in X_K(L)[T]\subset
X_K(L_E)[T]$ by its component in $E'$. Then by the linear
disjointness of $L|E$ and $E'|E$ it follows that
$f_E=f_{E'}$.

Now by Proposition~\ref{ApproxAPF}, $\nu_{E'}(f(\hat{y}^E)_{E'}-f_{E'}(\hat{y}^E_{E'}))\ge r(E').$
On the other hand, by (\ref{specialeq}) and the Ramification Argument (again applied to ${S(T)=\varpi g(Z,T)}$) we have
\begin{eqnarray*}
\nu_{E'}(f_E(\hat{y}^E_{E'})-F_E(\hat{y}^E_{E'}))&=&
\nu_{E'}(f_E(\hat{y}^E_{E'})-f_E(\hat{y}^E_{E'})-\pi_E^{r(E)}h_E(\hat{y}^E_{E'})\\
&&-\varpi g(\pi_E,\hat{y}^E_{E'}))\\
&=&\nu_{E'}(\pi_E^{r(E)}h_E(\hat{y}^E_{E'})+
\varpi g(\pi_E,\hat{y}^E_{E'}))\\
&\ge&\min\{r(E), \nu_{E}(\varpi)-B\}\\
&=&\min\{r(E), e(E|K)-B\},
\end{eqnarray*}
where $B$ is the maximal degree of the denominators in the coefficients of 
${g\in R[[Z]]_\mathfrak{m}[T]}$. 
Thus, we see that
\begin{eqnarray*}
\nu_{E'}(f(\hat{y}^E)_{E'}-F_E(\hat{y}^E_{E'}))&=&
\nu_{E'}(f(\hat{y}^E)_{E'}-
f_{E'}(\hat{y}^E_{E'})+f_{E'}(\hat{y}^E_{E'})
-F_E(\hat{y}^E_{E'}))\\
&=&\nu_{E'}(f(\hat{y}^E)_{E'}-f_{E'}(\hat{y}^E_{E'})
+f_E(\hat{y}^E_{E'})
-F_E(\hat{y}^E_{E'}))\\
&\ge&
\min\{r(E),e(E|K)-B\},
\end{eqnarray*}
where we have used the facts that $f_E=f_{E'}$ and $r(E')\ge r(E)$.

Consideration of the Taylor expansion of $F_E(T)$ at the point $T=y^E$, 
together with the fact that $\nu_{E'}(\hat{y}^E_{E'}-y^E)\ge r(E')$, shows that 
$\nu_{E'}(F_E(\hat{y}^E_{E'}))\ge r(E')$. Hence
$$
\nu_{E'}(f(\hat{y}^E)_{E'})=
\nu_{E'}(f(\hat{y}^E)_{E'}-F_E(\hat{y}^E_{E'})+F_E(\hat{y}^E_{E'}))\ge
\min\{r(E),e(E|K)-B\}.
$$
Thus we have shown that 
\begin{eqnarray*}
\nu_{X_K(L)}(f(\hat{y}^E))&\ge&
\frac{1}{\deg(f)}(\nu_{E'}(f(\hat{y}^E)_{E'}))
\ge\frac{1}{\deg(f)}\min\{r(E),e(E|K)-B\}.
\end{eqnarray*}
But $r(E)\rightarrow\infty$ for $E$ large, and since $B$ is a
constant, we also have ${e(E|K)-B\rightarrow\infty}$. It follows that 
$\nu_{X_K(L)}(f(\hat{y}^E))\rightarrow\infty$ so that
$\lim_{E\in\mathcal{C}_Y}f(\hat{y}^E)=0$
as claimed.~$\qed$
\end{proof}

\noindent
Replacing the net $\{\hat{y}^E\}_E$ by a subnet, we may assume that it
converges to a root $\tilde{\chi}$ of $f$. But then $\tilde{\chi}$ is conjugate to one
of the roots $\chi_j$ from the beginning of this proof, and since $\mathcal{K}|k((z))$ is Galois, we have that
$\mathcal{K}=k((z))(\chi_j)=k((z))(\tilde{\chi})$.\\

\noindent
F) By Krasner's Lemma,
$\tilde{\chi}\in X_K(L)(\hat{y}^E)\subset X_K(L_E)$ for $E\in\mathcal{C}_Y$ large. This implies that
$\mathcal{K}\subset X_K(L_E)$ for $E\in\mathcal{C}_Y$ large, and I claim that this inclusion
is actually an equality. For this we need a simple preliminary lemma.

Note that if $\sigma\in \textrm{Gal}(L_E|L)$, then
$X_{L|K}(\sigma)\in\textrm{Gal}(X_K(L_E)|X_K(L))$ and we have by definition
$$
X_{L|K}(\sigma)(\hat{y})=(\sigma(\hat{y}_B))_{B\in\mathcal{E}_{L_E|K}} \qquad \forall \hat{y}\in
X_K(L_E).
$$

\begin{lem}\label{GalNorm} 
Given $\sigma\in \textrm{Gal}(L_E|L)$, suppose that
$y\in E'$ is such that
${\nu_{E'}(\sigma(y)-y)<r(E').}$
Using Proposition~\ref{ApproxAPF}, choose an element $\hat{y}\in X_K(L_E)$ such that
$\nu_{E'}(\hat{y}_{E'}-y)\ge r(E').$
Then
$$
\nu_{X_K(L_E)}(X_{L|K}(\sigma)(\hat{y})-\hat{y})=
\nu_{E'}(\sigma(y)-y).
$$
\end{lem}

\begin{proof}
This follows from a straightforward computation using Proposition~\ref{ApproxAPF}.~$\qed$
\end{proof}

We wish to apply this lemma with $y=y^E$ and $\hat{y}=\hat{y}^E$, so we compute
\begin{eqnarray*}
\nu_{E'}(\sigma(y^E)-y^E)\le
\nu_{E'}(\textrm{disc}(F_E))
&\le&(\deg{F})\nu_{E}(\textrm{disc}(F_E))\\
&=&(\deg{F})\nu_{X_K(L)}(\textrm{disc}(f))
\hspace{1.7cm}(\dag)
\end{eqnarray*}
for $E$ large by Lemma~\ref{disc}.
Since $r(E')\rightarrow\infty$, it follows that $y^E$ satisfies the
hypothesis of Lemma~\ref{GalNorm} for $E$ large, and we conclude that
$\nu_{X_K(L_E)}(X_{L|K}(\sigma)(\hat{y}^E)-\hat{y}^E)
=\nu_{E'}(\sigma(y^E)-y^E)$. This immediately implies that $X_K(L)(\hat{y}^E)=X_K(L_E)$, because
if the inclusion were proper, then there would exist $\sigma\ne 1$ in
$\textrm{Gal}(L_E|L)$ such that $X_{L|K}(\sigma)(\hat{y}^E)=\hat{y}^E$,
which is a contradiction since $\sigma(y^E)\ne y^E$.

Thus, in order to show that $\mathcal{K}=X_K(L_E)$, we just need to show that
$X_K(L)(\hat{y}^E)\subset X_K(L)(\tilde{\chi})$. But the net $\{\hat{y}^E\}$
converges to $\tilde{\chi}$, and (\dag) shows that the Krasner
radii
$$
\max\{\nu_{X_K(L)}(X_{L|K}(\sigma)(\hat{y}^E)-\hat{y}^E) \ | 
\ \sigma\in G(L_E|L), \sigma\ne 1\}<C
$$
for some constant $C$ independent of $E$. Hence for $E$ sufficiently large so that
$\nu_{X_K(L)}(\tilde{\chi}-\hat{y}^E)>C$, Krasner's lemma tells us that
$X_K(L)(\hat{y}^E)\subset X_K(L)(\tilde{\chi})$
as required.

Thus, we have shown that $\mathcal{K}=X_{L|K}(L_E)$ for $E\in \mathcal{C}_Y$ large. It
now follows from the fundamental equality that $n_s=n_E$:
$$
n_s=\frac{\deg{f}}{[\mathcal{K}:k((z))]}=\frac{\deg{F}}{[L_E:L]}
=\frac{\deg{F}}{[E':E]}=n_E.
$$

It remains to prove the statement about the Galois groups. By the general theory
of the field of norms, we have
$$
\textrm{Gal}(L_E|L)\cong\textrm{Gal}(X_K(L_E)|X_K(L))
=\textrm{Gal}(\mathcal{K}|X_K(L)).
$$
Moreover, since $L_E=LE'$, and $L|E$ and $E'|E$ are linearly
disjoint, it follows that
$$
\textrm{Gal}(L_E|L)=\textrm{Gal}(LE'|L)\cong
\textrm{Gal}(E'|E'\cap L)=\textrm{Gal}(E'|E).
$$
Thus, we just need to show that the ramification filtrations are preserved under
these isomorphisms.

First note that for all $E,B\in\mathcal{C}_Y$ sufficiently large, we have $L_E=L_B$, since by the preceding
proof we have that $X_K(L_E)=\mathcal{K}=X_K(L_B)$ and $X_{L|K}(-)$ is an
equivalence of categories. Denote this common field by $L'$. 

\begin{lem}\emph{(compare \cite{W}, Proposition 3.3.2)}\label{RamBehavior}
For $\sigma\in \textrm{Gal}(L'|L)$ and $E$ large, we have 
$i_{E'}(\sigma)=i_{X_K(L')}(X_{L|K}(\sigma))$. 
\end{lem}

Since the lower ramification filtration is determined by the function $i$, it
follows that the isomorphism $\textrm{Gal}(E'|E)\cong\textrm{Gal}(L'|L)\cong
\textrm{Gal}(\mathcal{K}|X_K(L))$
induced by $X_{L|K}(-)$ preserves the ramification filtrations. Since the
degree of the different depends only on the ramification filtration, it
follows that $d_s=d_E$ for $E$ large. This completes the proof of part 1.\\

\noindent
{\it Proof of part 2}\\

\noindent
Note that by part 1, if $d_E\rightarrow\infty$, then the special fiber must
be generically inseparable, without any irreducibility hypothesis.

Now suppose that $Y_k$ is irreducible and $Y_k\rightarrow D_k$ is generically
inseparable. Let $V$ be the first ramification group at the unique prime of $\mathcal{A}$ lying over 
$(\varpi)$. Taking $V$-invariants, we obtain the tower
$Y\rightarrow Y^V\rightarrow D$. Since $V$ is a nontrivial $p$-group, it has a $p$-cyclic quotient. Hence $Y\rightarrow Y^V$ has a $p$-cyclic subcover $W\rightarrow Y^V$, and we have the tower $Y\rightarrow W\rightarrow Y^V\rightarrow D$. Now consider the associated tower of special fibers
$Y_k\rightarrow W_k\rightarrow Y_k^V\rightarrow D_k,$
which
corresponds (by considering the generic points) to a chain of field extensions
$k((z))\subset k((s))\subset k((x))\subset \mathcal{K}$. Here the extension $k((x))|k((s))$ is purely inseparable of degree $p$, defined by $x^p=s$.
Note that there is no extension of constants in this tower because the cover
$Y\rightarrow D$ was assumed to be regular.

The extension $k((s))|k((z))$ is separable and totally ramified, so the minimal polynomial of
$s$ over $k((z))$ is Eisenstein:
$$
g(T)=T^c+za_{d-1}(z)T^{c-1}+\cdots +za_{1}(z)T+zu(z)\in k[[z]][T],
$$
where $u(z)$ is a unit. It follows that $g(T^p)$ is the minimal polynomial of $x$ over $k((z))$. Now
let $\xi$ be a lifting of $x$ to the localized ring of global sections $\Gamma(W)_{(\varpi)}$. Then $\xi$ is integral over
$\mathcal{A}_{(\varpi)}$ and we let $G(T)\in \mathcal{A}_{(\varpi)}[T]$ 
be its minimal polynomial. Since $\deg(W | D)=\deg (k((x)) | k((z)))=pc$ and $\xi$ is a lifting of $x$, it follows that the degree of $G(T)$ is also $pc$, and $G(T)\equiv g(T^p)$ modulo $\varpi$. Using the Teichm\"uller lifting 
$\tau:k[[z]]\rightarrow R[[Z]]$ we find that:
\begin{eqnarray*}
G(T)&=&\tau(g)(Z, T^p)+\varpi P(Z,T)\\
&=&T^{pc}+Z\tau(a_{d-1})(Z)T^{p(c-1)}+\cdots +Z\tau(u)(Z)+\varpi P(Z,T),
\end{eqnarray*}
for some polynomial $P(Z,T)\in R[[Z]]_{(\varpi)}[T]$ of degree at most $pc-1$ in $T$.

Setting $Z=\pi_E$, we get the specialized polynomial in $E[T]$
$$
G_E(T)=T^{pc}+\pi_E\tau(a_{d-1})(\pi_E)T^{p(c-1)}+\cdots +\pi_E\tau(u)(\pi_E)
+\varpi P(\pi_E,T),
$$
which for $E$ sufficiently large is Eisenstein by the Ramification Argument applied to $S(T)=\varpi P(Z,T)$. Letting $\xi_E$ denote the image of $\xi$ in
$F(W_E)$, it follows  by degree considerations that $W_E$ is irreducible for $E$ large and $\xi_E$ is a uniformizer for the field  $F(W_E)$. 

We obtain the chain of field extensions
$E\subset E(\xi_E)=F(W_E)\subset E'$, and can compute the different as follows:
$$
\mathcal{D}(E(\xi_E) | E)=(G_{E}'(\xi_E))=(p\xi_E^{p-1}\tau(g)'(\pi_E, \xi_E^p)+\varpi
P'(\pi_E,\xi_E)).
$$
But
\begin{eqnarray*}
&\nu_{E(\xi_E)}&(p\xi_E^{p-1}\tau(g)'(\pi_E, \xi_E^p)+\varpi P'(\pi_E, \xi_E))
\ge\\
&&\min\{\nu_{E(\xi_E)}(p\xi_E^{p-1}\tau(g)'(\pi_E, \xi_E^p)),
\nu_{E(\xi_E)}(\varpi P'(\pi_E,\xi_E))\},
\end{eqnarray*}
and the latter quantity goes to $\infty$ with $E$. By
multiplicativity of the different in towers we conclude that
$$
d_E\ge \nu_{E(\xi_E)}(\mathcal{D}(E(\xi_E) | E)),
$$
so $d_E$ goes to $\infty$ with $E$ as claimed.\\

\noindent
{\it Proof of part 3}\\

\noindent
We now assume that $G$ is abelian, but make no separability assumption on 
the special fiber $Y_k\rightarrow D_k$.  Since $G$
is abelian, the decomposition groups at the $n_s$
primes of $\mathcal{A}_{(\varpi)}$ lying over $(\varpi)\in
\textrm{Spec}(R[[Z]])$ all coincide. Call this decomposition group
$\mathcal{Z}$. Taking $\mathcal{Z}$-invariants, we observe
that
$Y^\mathcal{Z}\rightarrow D$ is a $G/\mathcal{Z}$-Galois 
regular branched cover with totally split special fiber:
\begin{eqnarray}\label{isom}
F(Y_k^\mathcal{Z})\cong \prod_{j=1}^{n_s}k((z)).
\end{eqnarray}
In particular, there is no further splitting in the special fiber 
$Y_k\rightarrow Y_k^\mathcal{Z}$.
Hence, we can apply part~1  to the cover $Y^\mathcal{Z}\rightarrow D$, and a review of the beginning of the proof of part 1 shows that we may take $\mathcal{C}_{Y^\mathcal{Z}}=\mathcal{E}_{L|K}$. We
conclude that $n_E^\mathcal{Z}=n_s$ for $E$ large (here $n_E^\mathcal{Z}$ is the number of components of
$Y_E^\mathcal{Z}$). Since $Y_E\rightarrow Y_E^\mathcal{Z}$ is
surjective, it follows that $n_E\ge n_E^\mathcal{Z}=n_s$ as claimed. This completes the proof of part 3, and hence of Theorem~\ref{Main}.~$\qed$

\section{A Criterion for Good Reduction and the Oort Conjecture}
\label{sec:ArithmeticOort}

In this section, we use Theorem \ref{Main} to obtain a characterization of the abelian covers of the open $p$-adic disc having good reduction to characteristic $p$. For this, we need the following

\noindent
\begin{LocGR}\emph{(\cite{K} section 5, \cite{GM} 3.4)} Let
$\mathcal{A}$ be a normal integral local ring, which is also a finite
$R[[Z]]$-module. Assume moreover that $\mathcal{A}_s:=\mathcal{A}/\varpi\mathcal{A}$ is reduced
and $\textrm{Frac}(\mathcal{A}_s)|k((z))$ is separable. Let
$\widetilde{\mathcal{A}_s}$ be the integral closure of $\mathcal{A}_s$, and define
$\delta_k:=\dim_k(\widetilde{\mathcal{A}_s}/\mathcal{A}_s)$. Also, setting
$K=\textrm{Frac}(R)$, denote by
$d_\eta$ the degree of the different of $(\mathcal{A}\otimes K)|(R[[Z]]\otimes
K)$, and by $d_s$ the degree of the different of
$\textrm{Frac}(\mathcal{A}_s)|k((z))$. Then $d_\eta=d_s+2\delta_k$, and if
$d_\eta=d_s$, then $\mathcal{A}\cong R[[T]]$.
\end{LocGR}




Now let $H$ be a finite extension of $\mathbb{Q}_p$, and fix a Lubin-Tate extension
$L|K$ as described in section \ref{Coleman}, where $K:=H\widehat{\mathbb{Q}_p^{un}}$. Then choose a uniformizer $\pi=(\pi_E)_E\in X_K(L)$, which defines an
isomorphism $\overline{\mathbb{F}}_p((z))\cong X_K(L)$ as well as a net
of points $\{x^E\}_E\subset D_K$ (see section \ref{Disc}). We will be interested in the cofinal sequence of ramification subfields $\{L^m\}_m\subset \mathcal{E}_{L|K}$, and will use the simplified notation $\pi_m:=\pi_{L^m}, x_m=x^{L^m}, L_m=L_{L^m}$, etc.

\begin{prop}\label{arithcrit}
Suppose that $G$ is a finite abelian group, and $Y\rightarrow D$ is
a $G$-Galois regular branched cover with $Y$ normal and $Y_k$ reduced. Then
$Y\rightarrow D$ has good reduction (with $Y_k$ irreducible and $Y_k\rightarrow D_k$ separable), if and only if there
exists a $G$-Galois extension $M|L$ and an integer $l>0$ such that for $m>>0$ and $m\equiv 1 \mod l$, we have $L_m=M$ as $G$-Galois extensions of $L$,
and $d_m=d_\eta$. In this case, the generic fiber of $Y_k\rightarrow D_k$ corresponds to the field extension $X_K(M)|X_K(L)$.
\end{prop}

\begin{proof}
First suppose that $Y\rightarrow D$ has good reduction with $Y_k$ irreducible and 
$Y_k\rightarrow D_k$ separable. Then by the proof of part
1 of Theorem \ref{Main} there exists $l>0$ such that for $m>>0$ 
and $m\equiv 1 \mod l$, we have $F(Y_k)=X_K(L_m)$ and $d_m=d_s$.
Since $X_K(-)$ is an equivalence of categories, we conclude that for these values of $m$, the fields $L_m$ are all equal. Let $M|L$ be this common $G$-Galois extension.
By the Local Criterion For Good Reduction,
we have $d_s=d_\eta$, which
implies that $d_m=d_\eta$ for $m>>0$ and $m\equiv 1 \mod l$, as claimed.

Now suppose that there exists $l>0$ so that 
$L_m=M$ and $d_m=d_\eta$ for $m>>0$ and $m\equiv 1 \mod l$. Then by part 3 of 
Theorem \ref{Main},
$Y_k$ is irreducible, and then by part 2, $Y_k\rightarrow D_k$ is separable. Hence we may apply part 1 to
conclude that there exists $l_1>0$ such that 
$F(Y_k)=X_K(L_m)$ and $d_s=d_m$ for $m>>0$ and $m\equiv 1
\mod l_1$. But the two arithmetic progressions $\{tl+1\}_t$ and $\{tl_1+1\}_t$
have a common subsequence. It follows that $F(Y_k)=X_K(M)$ and $d_s=d_\eta$, so
$Y\rightarrow D$ is a birational lifting of $X_K(M)|X_K(L)$ which preserves the
different. By
the Local Criterion for Good Reduction, it follows that $Y\rightarrow D$ is actually a
smooth lifting.~$\qed$
\end{proof}

As an application, we obtain an arithmetic reformulation of the Ring Specific Local Oort Conjecture from the Introduction
concerning the liftability of cyclic covers over $\overline{\mathbb{F}}_p$. Set
$K=\widehat{\mathbb{Q}_p^{un}}$, and let $L=K(\zeta_{p^\infty})$. Then $L|K$ is
Lubin-Tate for $H=\mathbb{Q}_p$ and $\Gamma=\widehat{\mathbb{G}_m}$. Moreover, if $C$
is a finite cyclic group, define
$R_C:=R_K[\zeta_{|C|}]\subset R_L$. 

\begin{ArithOortFp}
Suppose that $M | L$ is a finite cyclic
extension of $L$, with group $C$. 
Then there exists $l>0$ and a normal, $C$-Galois, regular branched cover
$Y\rightarrow D:=\textrm{Spec}(R_C[[Z]])$ such that
\begin{enumerate}
\item $Y_k$ is reduced;
\item $L_m=M$ for $m>>0$ and $m\equiv 1 \mod l$;
\item $d_\eta=d_m$ for $m>>0$ and $m\equiv 1 \mod l$.
\end{enumerate}
\end{ArithOortFp}

\begin{prop}
The Arithmetic Form of the Ring Specific Local Oort Conjecture is equivalent to the Ring Specific Local Oort Conjecture over $\overline{\mathbb{F}}_p$.
\end{prop}

\begin{proof} First assume that the Arithmetic Form holds, and 
suppose that ${W_k\cong \textrm{Spec}(k[[t]])\rightarrow D_k}$ is a $C$-Galois cover, corresponding to the field extension $F(W_k)|X_K(L)$. Since $X_K(-)$ is an equivalence of categories, there exists a unique $C$-Galois extension $M|L$ such that $F(W_k)=X_K(M)$. Let $Y\rightarrow D$ be the $C$-Galois cover furnished by the Arithmetic Form of the conjecture. Then by Proposition \ref{arithcrit}, $Y\rightarrow D$ is a smooth lifting of $W_k\rightarrow D_k$, and hence the Ring Specific Local Oort Conjecture holds. 

Conversely, assume that the Ring Specific Local Oort Conjecture holds, and suppose that $M|L$ is a $C$-Galois extension. Apply the field of norms to obtain a $C$-Galois extension $X_K(M)|X_K(L)$, corresponding to a $C$-Galois cover $Y_k\rightarrow D_k$ with $Y_k\cong \textrm{Spec}(k[[t]])$. Let $Y\rightarrow D$ be a smooth lifting of $Y_k$. Then by Proposition $\ref{arithcrit}$, conditions 1-3 of the Arithmetic Form are satisfied for the cover $Y\rightarrow D$, so the Arithmetic Form of the Ring Specific Local Oort Conjecture holds.~$\qed$
\end{proof}

\begin{rem}
It can be shown by standard techniques of model theory that the Ring Specific Local Oort Conjecture over $\overline{\mathbb{F}}_p$ implies the Ring Specific Local Oort Conjecture over $k$, where $k$ is an arbitrary algebraically closed field of characteristic $p$. 
\end{rem}

\begin{rem}
One can give a direct proof of the Arithmetic Form of the Ring Specific Local Oort Conjecture for $p$-cyclic covers over $\overline{\mathbb{F}}_p$. It is a variant of the proofs given in \cite{OSS} and \cite{GM}, using Kummer Theory in characteristic zero in place of Artin-Schreier Theory in characteristic $p$.
\end{rem}

\begin{acknowledgement}
I would like to thank Florian Pop for suggesting this line of inquiry for my Ph.D. dissertation at the University of Pennsylvania. In addition, I am grateful to the anonymous referree for many helpful comments.
\end{acknowledgement}

\noindent
Scott Corry \quad corrys@lawrence.edu\\
\noindent
Lawrence University, Department of Mathematics, 711 E. Boldt Way, Appleton, WI 54911, USA

\end{document}